\documentclass[twoside]{article}

\usepackage{vmargin}           
\usepackage{fancyhdr}           
\usepackage{ifthen} 

\usepackage{palatino}   
\usepackage[cm]{sfmath}  

\usepackage{amsmath}   
\usepackage{amssymb}   
\usepackage{amsfonts}
\usepackage{amsthm}
\usepackage[bbgreekl]{mathbbol} 

\usepackage{pst-node} 
\usepackage{pstricks}

\usepackage{tweaklist} 
\renewcommand{\itemhook}{\setlength{\topsep}{1pt}%
\setlength{\itemsep}{0pt}}

\setpapersize{A4}
\setmarginsrb{33mm}{20mm}{33mm}{25mm}{10mm}{6mm}{0mm}{10mm}

\newcommand{\sk}{\smallskip}
\newcommand{\mk}{\medskip}
\newcommand{\bk}{\bigskip}
\renewcommand{\emptyset}{\ensuremath{\varnothing}}     

\newlength{\leftlength}
\newlength{\rightlength}
\newlength{\calculskip}
\newcommand{\calculvskip}[1]{%
  \ifthenelse{#1 = 0}{\setlength{\calculskip}{0pt}}{}%
  \ifthenelse{#1 = 1}{\setlength{\calculskip}{\smallskipamount}}{}%
  \ifthenelse{#1 = 2}{\setlength{\calculskip}{\medskipamount}}{}%
  \ifthenelse{#1 = 3}{\setlength{\calculskip}{\bigskipamount}}{}%
  \ifthenelse{#1 = 4}{\setlength{\calculskip}{1cm}}{}%
  \vskip\calculskip
}

\newcommand{\leftcentersright}[4][2]{%
        \settowidth{\leftlength}{#2}%
        \settowidth{\rightlength}{#4}%
        \calculvskip{#1}
        \noindent#2\hskip-\leftlength%
        \hfill#3\hfill
        \mbox{}\hskip-\rightlength#4%
        \vskip\calculskip%
        }

\newcommand{\centers}[2][2]{\leftcentersright[#1]{}{#2}{}}
\newcommand{\leftcenters}[3][2]{\leftcentersright[#1]{#2}{#3}{}}

\makeatletter
\renewcommand\section{\@startsection{section}{1}{\z@}%
	{1cm  \@plus -1ex \@minus -.2ex}%
	{2.3ex \@plus.2ex}%
	{\reset@font\normalsize\scshape\centering}}
\makeatother

\newtheorem{thm}{Theorem}[section]  

\newtheorem{prop}[thm]{Proposition}
  
\newtheorem{cor}[thm]{Corollary}
\newtheorem{qu}[thm]{Question}
\newtheorem{for}[thm]{Formula}

\newenvironment{rmk}{
	\refstepcounter{thm}%
	\noindent \textbf{Remark \thethm.}}{}
\newenvironment{example}{
	\refstepcounter{thm}%
	\noindent \textbf{Example \thethm.}
}{}
\newenvironment{notations}{
	\refstepcounter{thm}%
	\noindent \textbf{Notations \thethm.}
}{}

\newcommand{\bb}{Bialynicki-Birula}

\begin{document}

\title{Note on the Deodhar decomposition \\ of a double Schubert cell}
\author{Olivier Dudas\footnote{Laboratoire de Math\'ematiques, Universit\'e de Franche Comt\'e}\ \,\footnote{The author is partly supported by the ANR, Project No JC07-192339.}}

\maketitle

\begin{abstract} We show that for an algebraic reductive group $G$, the partition of a double Schubert cell in the flag variety $G/B$ defined by Deodhar, and coming from a \bb\ decomposition, is not a stratification in general. We give a counterexample for a group of type B$_n$, where the closure of  some specific cell of dimension $2n$ has a non-trivial intersection with a cell of dimension $3n-3$. \end{abstract}

\section*{Introduction}

Let $G$ be an algebraic reductive group defined over an algebraically  closed field $k$ together with a fixed Borel subgroup $B$ containing a maximal torus $T$ of $G$. The Coxeter system corresponding to these data will be denoted by $(W,S)$. More precisely, $W = N_G(T)/T$ and $S$ is the set of non-trivial elements $s \in W$ such that $BsB$ is of minimal dimension.  The opposite Borel subgroup $B^*$ will be defined as the conjugate of $B$ by the longest element $w_0$ of $W$. 

\mk

We will be concerned with a refinement of the Bruhat stratification of the flag variety $G/B$. Recall that under the action of $B$ (resp. $B^*$), this variety decomposes into a disjoint union of orbits, each of them containing a unique element of $W$. Such an orbit will be denoted by $Bw \cdot B$ (resp. $B^*w \cdot B$) and referred as the Schubert cell (resp. the opposite Schubert cell) corresponding to $w$.

\mk

Given two elements of the Weyl group $w$ and $v$, Deodhar has defined in \cite{Deo}  a partition of the 
double Schubert cell $Bw\cdot B \cap B^*v \cdot B$ into affine smooth locally closed subvarieties of the flag variety $G/B$. This decomposition is not unique in general and depends on a reduced expression of $w$. When such an expression is chosen, the decomposition has a combinatorial definition: the set of cells is parametrized by some subexpressions of $w$, the distinguished ones, and each cell is isomorphic to $k^n \times (k^\times)^m$ where $n$ and $m$ can be defined in terms of the associated subexpression (see \cite[theorem 1.1]{Deo}).

\sk

In the special case where $w$ is a Coxeter element, Deodhar was able to describe the closure of a cell (see \cite[section 4]{Deo}), giving thus a complete description of the geometry of the double Schubert cell. This particular example, together with the recent work of Webster and Yakimov on a more general decomposition (see \cite{WY} and \cite{We}), lead to the following expectations: 

\begin{itemize} \renewcommand{\labelitemi}{(i)}

\item the closure of a cell is a union of cells; \vskip-19pt $\mathrm{ }$

 \renewcommand{\labelitemi}{(ii)}

\item there is a natural order on the set of cells related to the Bruhat order, such that the closure of a cell has a non trivial intersection with all the smaller cells for this order.

\end{itemize}

\noindent Unfortunately, these two assertions fail in general, and we give two examples showing that the situation is much more complicated (section \ref{proof}.2 and \ref{proof}.3). At the present time, we have no clue for what can be the closure of a cell.

\section*{Acknowledgments}

This work was carried out while I was a Program Associate at the Mathematical Sciences Research Institute in Berkeley. I wish to thank the Institution for their support and hospitality and especially Arun Ram, who really helped me and the other graduate students to make the most of our stay.  I also thank C\'edric Bonnaf\'e, Geordie Williamson and St\'ephane Gaussent for fruitful discussions about geometry of Schubert varieties.

\section{Double Schubert cells and Deodhar decomposition}

We recall in this section the principal result of \cite{Deo}, using a different approach due to Morel (see \cite[Section 3]{Mo}) which relies on a general decomposition theorem, namely the \bb\ decomposition, applied to Bott-Samelson varieties. 

\bk

Let $w \in W$ be an element of the Weyl group of $G$.  The Schubert variety $X_w$ associated to $w$ is the closure in $G/B$ of the Schubert cell $Bw\cdot B$. This variety is not smooth in general, but Demazure has constructed in \cite{Dem} a resolution of the singularities, called the Bott-Samelson resolution, which is a projective smooth variety over $X_w$. The construction is as follows: we fix a reduced expression $w=s_1 \cdots s_\ell$ of $w$ and we define the Bott-Samelson variety to be

\centers{$BS = P_{s_1} \times_B \cdots \times_B P_{s_\ell} /B$}

\noindent where $P_{s_i}= B \cup Bs_iB$ is the standard parabolic subgroup corresponding to the simple reflection $s_i$. It is thus defined as the quotient of $P_{s_1} \times \cdots \times P_{s_\ell}$  by the right action  of $B^\ell$ given by $(p_1, \ldots,p_\ell) \cdot (b_1, \ldots, b_\ell) = (p_1 b_1 , b_1^{-1}p_2 b_2 , \ldots,p_{\ell -1}^{-1}p_\ell b_\ell)$.
The homomorphism $\pi : BS \longrightarrow X_w$ which sends the class  $[p_1,\ldots,p_\ell]$ in $BS$ of an element  $(p_1,\ldots,p_\ell) \in  P_{s_1} \times \cdots \times P_{s_\ell} $ to the class of the product $p_1 \cdots p_\ell$ in $G/B$ is called the \textbf{Bott-Samelson resolution}. It is a proper surjective morphism of varieties and it induces an isomorphism between $\pi^{-1}(Bw\cdot B)$ and $Bw\cdot B$.

\sk

Now the torus $T$ acts naturally on $BS$ by left multiplication on the first component, or equivalently by conjugation on each component,  so that $\pi$ becomes a $T$-equivariant morphism. There are finitely many fixed points for this action, represented by the classes of the elements of $\Gamma = \{1,s_1\} \times \cdots \times \{1,s_\ell\}$ in $BS$; such an element will be called a \textbf{subexpression} of~$w$. 

\sk

For a subexpression $\gamma = (\gamma_1,\ldots, \gamma_\ell) \in \Gamma$ of $w$, we denote by $\gamma^i = \gamma_1 \cdots \gamma_i$ the $i$-th partial subword and we define the following two sets:

\leftcenters{and}{ $\begin{array}[b]{r@{\ \, = \ \, }l} 
I(\gamma) & \big\{i \in \{1, \ldots, \ell \} \ | \ \gamma_i = s_i \big\} \\[4pt]
J(\gamma) & \big\{i \in \{1,\ldots, \ell \} \ | \ \gamma^i s_i < \gamma^i \big \}. \end{array}$}

\noindent With these notations,  Deodhar's decomposition theorem (see \cite[Theorem 1.1 and Corollary 1.2]{Deo}) can be stated as follows:

\begin{thm} [Deodhar, 84]\label{ddec}There exists a family  $(D_\gamma)_{\gamma \in \Gamma}$ of disjoint smooth locally closed subvarieties of $Bw\cdot B$ such that: \begin{itemize}

\item[$\mathrm{(i)}$]  $D_\gamma$ is non empty if and only if $J(\gamma) \subset I(\gamma)$;

\item[$\mathrm{(ii)}$]  if  $D_\gamma$ is non empty, then it is isomorphic to $k^{|I(\gamma)|- |J(\gamma)|} \times (k^\times)^{\ell-|I(\gamma)|}$ as a variety;

\item[$\mathrm{(iii)}$] for all $v \in W$, the double Schubert cell has the following decomposition:

\centers{$ \displaystyle Bw \cdot B \cap B^* v \cdot B = \coprod_{\gamma \in \Gamma_v} D_\gamma$}

\noindent where $\Gamma_v$ is the subset of $\Gamma$ consisting of all subexpressions $\gamma$ such that $\gamma^\ell = v$.
\end{itemize} 

\end{thm}

\mk

\begin{rmk} In the first assertion, the condition for a cell $D_\gamma$ to be non-empty, that is $J(\gamma) \subset I(\gamma)$,  can be replaced by:

\centers{$\forall \, i=2, \ldots, \ell \qquad  \gamma^{i-1} s_i < \gamma^{i-1} \ \Longrightarrow \ \gamma_i = s_i$.}

\noindent A subexpression $\gamma \in \Gamma$ which satisfies this condition is called a \textbf{distinguished subexpression}. For example, if $G = \mathrm{SL}_3(k)$ and $w=w_0 = s t s$, then there are seven distinguished subexpressions, the only one being not distinguished is $(s,1,1)$.
\end{rmk}

\begin{proof}[Sketch of proof:] the Bott-Samelson variety is a smooth projective variety endowed with an action of the torus $T$. Let us consider the restriction of this action to $\mathbb{G}_m$ through a strictly dominant cocharacter $\chi : \mathbb{G}_m \longrightarrow T$. Since this action has a finite number of fixed points, namely the elements of $\Gamma$, there exists a Bialynicki-Birula decomposition of the variety $BS$ into a disjoint union of affine spaces indexed by $\Gamma$ (see \cite[Theorem 4.3]{BB}) 

\centers{$ \displaystyle BS = \coprod_{\gamma \in \Gamma} C^\gamma$.}

\noindent In \cite{Ha}, H\"arterich has explicitly computed the cells $C^\gamma$. To describe this computation, we need some more notations: $\Phi$ will be the root system corresponding to the pair $(G,T)$ and $\Phi^+$ (resp. $\Phi^-$) the set of positive (resp. negative) roots defined by $B$ (resp. $B^*$).  For any root $\alpha \in \Phi$ we denote by $U_\alpha$ the corresponding one-parameter subgroup and we choose an isomorphism $u_\alpha : k \longrightarrow U_\alpha$.  The simple roots associated to the simple reflections of the reduced expression $w = s_1 \cdots s_\ell$ will be denoted by $\alpha_1,\ldots,\alpha_\ell$. Finally, we consider the open immersion $a_\gamma : \mathbb{A}_\ell \longrightarrow BS$ defined by 

\centers{ $ a_\gamma(x_1,\ldots,x_\ell) = [u_{\gamma_1(-\alpha_1)}(x_1) \gamma_1, \ldots, u_{\gamma_\ell(-\alpha_\ell)}(x_\ell) \gamma_\ell ]. $}

\noindent Then one can easily check that $\pi^{-1}(Bw \cdot B) = \mathrm{Im}(a_{(s_1,\cdots,s_r)})$. Moreover, H\"arterich's computations (see \cite[Section 1]{Ha}) show that for any subexpression $\gamma \in \Gamma$, one has: 

\centers{$ C^\gamma \, = \, a_\gamma\big(\{ (x_1, \ldots,x_\ell) \in \mathbb{A}_\ell \ | \ x_i = 0 \ \text{ if }\ i \in J(\gamma) \} \big).$}

\noindent Taking the trace of this decomposition with $\pi^{-1}(Bw\cdot B)$, one obtains a decomposition of the variety $\pi^{-1}(Bw\cdot B)$. Furthermore, the restriction of $\pi$ to this variety induces an isomorphism with $Bw\cdot B$, and thus gives a partition of $Bw\cdot B$ into disjoint cells:

\centers{$ \pi^{-1}(Bw\cdot B) \, = \, \displaystyle \coprod_{\gamma \in \Gamma}  \pi^{-1}(Bw\cdot B) \cap C^\gamma \, \simeq \, \coprod_{\gamma \in \Gamma}  Bw\cdot B \cap \pi(C^\gamma) \, = \, Bw\cdot B.$}

\noindent If we define $D_\gamma$ to be the intersection $Bw\cdot B \cap \pi(C^\gamma)$, then it is explicitly given by:

\centers{$ D_\gamma \simeq \pi^{-1}(D_\gamma) \, = 
\, a_\gamma\big(\{ (x_1, \ldots,x_\ell) \in \mathbb{A}_\ell \ | \ x_i = 0 \ \text{ if }\ i \in J(\gamma) \ \text{ and } \ x_i \neq 0 \ \text{ if } \ i \notin I(\gamma) \} \big).$}

\noindent This description, together with the inclusion $\pi(C^\gamma) \subset B^* \gamma^\ell \cdot B$, proves the three assertions of the theorem.

\end{proof}

\begin{example}\ \label{ex}In the case where $G = \mathrm{SL}_3(k)$, and $w = w_0 = sts$, one can easily describe the double Schubert cell $Bw\cdot B \cap B^*\cdot B$. It is isomorphic to $BwB \cap U^*$ by the map $u \mapsto uB$, where $U^*$ denotes the unipotent radical of $B^*$. Besides, by Gauss reduction, the set $Bw Bw^{-1} = BB^*$ consists of all matrices whose principal minors are non-zero. Hence, 

\centers{$ BwB \cap U^* \, = \, \left\{ \left( \begin{array}{ccc} 1 & 0 & 0 \\ a& 1 & 0 \\ c & b & 1 \end{array} \right) \Big| \ c \neq 0 \ \text{ and } \ ab-c \neq 0 \right\}.$}

\noindent Considering the alternative $a = 0$ or $a \neq 0$, one has $BwB \cap U^* \simeq (k^\times)^3 \cup k \times k^\times$, which is exactly the decomposition given by the two distinguished expressions $(1,1,1)$ and $(s,1,s)$.\end{example}

\mk

\begin{notations}\ \label{not}For a subexpression $\gamma \in \Gamma$, we define the sequence

\centers{$\Phi(\gamma) \, = \, \big(\gamma^i(-\alpha_i) \ \big| \  i=1,\ldots, \ell \ \text{and} \ \gamma^i(\alpha_i) > 0\big).$}

\noindent Using H\"arterich's computation for the cell $C^\gamma$ and the definition of $\pi$, one can see that each element of $\pi(C^\gamma) \subset B^* \gamma^\ell \cdot B$ has a representative in the unipotent radical $U^*$ of $B^*$ which can be written in the following form:

\centers{$\displaystyle \prod_{\alpha \in \Phi(\gamma)} u_\alpha (x_\alpha) \qquad \text{with each} \ \ x_\alpha \in k$,}

\noindent the product being taken with respect to the order on $\Phi(\gamma)$. At the level of $D_\gamma$, some of the variables $x_\alpha$ must be non-zero (those corresponding to $\gamma^i(- \alpha_i)$ with $\gamma_i = 1$) but the expression becomes unique, and it will be referred as the \textbf{canonical expression} in $U^*$ of an element of $D_\gamma$. \end{notations}

\section{On the closure of  Deodhar cells\label{proof}}

This section is devoted to the two questions raised in the introduction. Before recalling them, we make the statements more precise. For $w = s_1 \ldots s_\ell$ a reduced expression of an element $w$ of $W$, we have defined in the previous section a desingularization of the Schubert variety $X_w$. One can embed this variety into a product of flag varieties as follows: we define the morphism $\iota : BS \longrightarrow (G/B)^\ell$ by

\centers{$  \iota([p_1,p_2, \ldots,p_\ell]) =(p_1B,p_1p_2B, \ldots, p_1 p_2 \cdots p_\ell B).$ }

\noindent Note that $\pi$ is the last component of this morphism. Let $\gamma \in \Gamma$ be a subexpression of $w$. As a direct consequence of the construction of $C^\gamma$, one has 

\centers{$\iota(C^\gamma) \ \subset \ \displaystyle \prod_{i=1}^\ell B^*\gamma^i \cdot B. $}

\noindent Since $BS$ is projective, $\iota$ is a closed morphism, and hence it sends the closure of a cell $C^\gamma$ in $BS$ to the closure of $\iota(C^\gamma)$. Therefore, it is natural to consider a partial order on the set $\Gamma$ coming from to the Bruhat order on $W$ since it describes the closure relation for Schubert cells. For $\delta \in \Gamma$, we define

\centers{$ \delta \preceq \gamma \ \ \iff \ \ \gamma^i \leq \delta^i\ $ for all $\ i = 1, \dots, \ell.$}

\leftcenters{Then, by construction:}{ $\overline{C^\gamma} \ \subset \ \displaystyle \bigcup_{\delta \preceq \gamma} C^\delta$ \quad and \quad $\overline{D_\gamma} \ \subset \ \displaystyle \bigcup_{\delta \preceq \gamma} D_\delta$}

\noindent where $\overline{D_\gamma}$ denotes the closure of $D_\gamma$ in the Schubert cell $Bw\cdot B$.  Now with these notations, the questions raised in the introduction can be rewritten as:

\begin{qu}\ \label{q1}Is the closure of $D_\gamma$ a union of cells ? In other terms, does the partition $(D_\gamma)_{\gamma \in \Gamma}$ define a stratification of the variety $B w \cdot B$ ?
\end{qu}

\begin{qu}\ \label{q2}For a subexpression $\delta \preceq \gamma$, do we have $\overline{D_\gamma} \cap D_\delta \neq \emptyset$ ?
\end{qu}

It is possible to give a positive answer to both of these questions in some specific cases -$w$ a Coxeter element or $\gamma$ maximal. However, this is not the case in general, and the situation can be even worse, as shown in the following sections.

\bk\mk

\noindent \textbf{II.1 - Chevalley formula in type B$_n$}\mk

From now on, $G$ will be a quasi-simple group of type B$_n$, for example the orthogonal group $\mathrm{SO}_{2n+1}(k)$. The Weyl group $W=W_n$ and its underlying root system correspond to the following Dynkin diagram: \mk

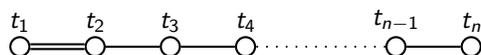
\centers{\begin{pspicture}(6,0)

  \cnode(0,0.05){4pt}{A}
  \cnode(1,0.05){4pt}{B}
  \cnode(2,0.05){4pt}{C}
  \cnode(3,0.05){4pt}{D}
  \cnode(5,0.05){4pt}{E}
  \cnode(6,0.05){4pt}{F}

  \ncline[nodesep=0pt,doubleline=true]{A}{B}
  \ncline[nodesep=0pt,linecolor=white]{A}{B}\naput[npos=-0.2]{$t_1$} \naput[npos=1.2]{$t_2$}
  \ncline[nodesep=0pt]{B}{C}
  \ncline[nodesep=0pt]{C}{D}\naput[npos=-0.2]{$t_3$} \naput[npos=1.2]{$t_4$}
  \ncline[nodesep=0pt,linestyle=dotted]{D}{E}
  \ncline[nodesep=0pt]{E}{F}\naput[npos=-0.2]{$t_{n-1}$} \naput[npos=1.2]{$t_n$}

\end{pspicture}}

\sk

\noindent The set of generators will be denoted by $S = \{t_1,\ldots,t_n\}$ and the associated simple roots by $\{\beta_1,\ldots,\beta_n\}$. There are $n^2$ positive roots, and their expression in terms of the simple ones is given by \cite[Planche II]{Bou}:
\begin{itemize}

\item[$\bullet$] $\alpha_i + \alpha_{i+1} + \cdots + \alpha_j$ \ for $1\leq i \leq j \leq n$; 

\item[$\bullet$] $2\alpha_1 + \cdots + 2 \alpha_i + \alpha_{i+1} + \cdots + \alpha_j$ \ for $1\leq i < j \leq n$.

\end{itemize}

\sk

Recall that to each of these roots and their opposite correspond a one-parameter subgroup $u_\alpha : k \longrightarrow U_\alpha$. Since every element of a Deodhar cell can be written in terms of these subgroups (see notations \ref{not}), we need to  recall the fundamental tool we will be using for all the computations, that is, the Chevalley commutator formula (see \cite[Theorem 5.2.2]{Car}). One may, and we will, choose indeed the family $(u_\alpha)_{\alpha \in \Phi}$ such that if $\alpha, \beta \in \Phi$ are any linearly independent roots and $x,y \in k$ any scalars, one has:

\centers{$\big[u_\alpha(x) \, ; u_\beta(y)\big] \, = \, u_\alpha(x) u_\beta(y) u_\alpha(-x) u_\beta(-y) \, = \, \displaystyle \prod_{i,j >0} u_{i\beta + j \alpha}(C_{ij\beta\alpha} \, (-y)^i x^j)$}

\noindent where the product is taken over all pairs of positive integers $i,j$ for which $i\beta + j\alpha$ is a roots, in order of increasing $i+j$. For the simplicity of the proofs, we give here some explicit expressions of this formula in the specific cases we will encounter: 

\begin{for}\label{for}Let $x,y \in k$. For  $\alpha, \beta \in \Phi^-$ and $i = 2, \ldots,n-1$, we have \begin{itemize}
\item[\emph{(i)}] if $\alpha + \beta \notin \Phi$ \ then \
{$u_\alpha(x) u_\beta(y) u_\alpha(-x) \, = \, u_\beta(y)$ ;}

\item[\emph{(ii)}] if $\alpha = -\beta_i$ and $\beta = -\beta_{i+1} - \cdots - \beta_{n} $   \ then \ 
{$u_ \alpha(x) u_{\beta}(y) u_{\alpha}(-x) \, = \, u_{\alpha + \beta}(\pm xy) u_{\beta}(y) $ ;}

\item[\emph{(iii)}] if $\alpha = -2\beta_1 - \beta_2 - \cdots - \beta_{n-1}$ and $\beta = -\beta_2 - \cdots -\beta_n$ \ then 
{$u_ \alpha(x) u_{\beta}(y) u_{\alpha}(-x) \, = \, u_{\alpha + \beta}(\pm xy) u_{\beta}(y) $ ;}

\item[\emph{(iv)}] if $\alpha = -\beta_i-\cdots -\beta_{n-1}$ and $\beta = -\beta_{n} $   \ then \
{$u_ \alpha(x) u_{\beta}(y) u_{\alpha}(-x) \, = \, u_{\beta}(y) u_{\alpha + \beta}(\pm xy) $ ;}

\item[\emph{(v)}] if $\alpha = -\beta_1-\cdots -\beta_{n-1}$ and $\beta = -\beta_{n} $  \ then \
{$\big[u_ \alpha(x)\, ; u_{\beta}(y)\big]  \, = \,  u_{2\alpha + \beta}(\pm x^2 y) u_{\alpha + \beta}(\pm xy)  $.}

\end{itemize}
\end{for}

\begin{rmk} The values of the constants $C_{ij\beta\alpha}$ can be determined by  \cite[Section 4.3]{Car}. Note that the signs of these constants depend on a choice on some of the elements of the Chevalley basis of the Lie algebra of $G$ (namely, the extra-special pairs, see \cite[Section 4.2]{Car}). However, this will not be relevant in our computations and we will use the notation $\pm$. \end{rmk}

\bk\mk

\noindent \textbf{II.2 - Obstruction to the stratification} \mk

In this section we give a negative answer to question \ref{q1}. To do so, we consider an element $w$ of $W_n$ defined by the following reduced expression:

\centers{$ w=t_n t_{n-1} \cdots t_2 t_1t_2 \cdots t_{n-1} t_n t_{n-1} \cdots t_2 t_1 t_2 \cdots t_{n-1}$}

\noindent and we define $\gamma,\delta \in \Gamma$ to be the following two distinguished subexpressions of $w$:

\leftcenters{and}{$\begin{array}[b]{r@{\, \ = \, \ }l} 
 \gamma & (1,t_{n-1},t_{n-2}, \ldots,t_2,1,t_2, \ldots, t_{n-1}, 1, t_{n-1}, \ldots, t_2,1,t_2, \ldots, t_{n-1}) \\[4pt]
\delta & (1,t_{n-1},t_{n-2}, \ldots,t_2,t_1,1,1,\ldots \ldots \ldots \ldots \ldots \ldots \ldots ,1,t_1,t_2, \ldots,t_{n-1}).
\end{array}$}

\noindent The dimension of the cells associated to these subexpressions is given by theorem \ref{ddec}.(ii). One can easily check that $\dim D^\gamma = 2n$ and $\dim D^\delta = 3n-3$ although the two subexpressions are related by $\delta \preceq \gamma$. Therefore, for $n \geq 4$, the closure of $D^\gamma$ cannot contain the cell $D^\delta$ and in this situation, one can no longer give a positive answer to both of the questions. More precisely, we prove:

\begin{prop}\label{prop1}The closure of $D^\gamma$ in the double Schubert cell $Bw\cdot B \cap B^* \cdot B$ contains a subvariety of $D^\delta$ of dimension $n$. 
\end{prop}

\begin{proof}  (i) Let $\Psi$ be the subset of the root system $\Phi$ defined by 

\centers{$ \Psi = \{-2\beta_1 - \cdots - 2 \beta_{n-1} - \beta_n \, ; -\beta_2 - \cdots - \beta_n\, ; -\beta_3 - \cdots - \beta_n \, ; \ldots ; -\beta_{n-1} - \beta_n \, ; -\beta_n \}. $}

\noindent The sum of two elements of this subset is never a root, so that all the corresponding one-parameter subgroups commute. Associated to this set of roots, we define 

\centers{$ V \, = \, \displaystyle \prod_{\beta \in \Psi} u_\beta(k^\times) \ \subset \ U^*$.}

\noindent By the previous remark and formula \ref{for}.(i), this product does not depend on any order on $\Psi$. In order to make the connection with the cells $D_\gamma$ and $D_\delta$, we define the corresponding variety in $G/B$ by 
 
 \centers{$\Omega \, = \, V \cdot B \ \subset \ B^* \cdot B$.}
 
 \noindent It is an affine variety of dimension $n$, isomorphic to $V$. We show now that it is contained in both $\overline{D_\gamma}$ and $D_\delta$, which will prove the assertion of the theorem.

\sk

\noindent (ii) Using \cite[Section V.4.1]{Bou}, one can easily determine the elements of the sequence $\Phi(\delta)$; their opposite are given by 

\centers{$-\Phi(\delta) \, = \, \big( \hskip -1.3mm \begin{array}[t]{l} \beta_n \, ; 2\beta_1 + \beta_2 + \cdots + \beta_{n-1}\, ; \beta_2\, ; \beta_ 3\, ; \ldots ; \beta_{n-2}\, ; \beta_{n-1} + \beta_n \, ; \beta_{n-2}\, ; \ldots ; \beta_3\, ; \beta_2\, ; \\[3pt] 2\beta_1 + \beta_2 + \cdots + \beta_{n-1} \, ;  \beta_1 + \cdots + \beta_{n-1}\, ; \beta_2 + \cdots + \beta_{n-1}\, ; \ldots  ; \beta_{n-1}\big). \end{array}$}

\noindent Recall from notations \ref{not} that the elements of $D_\delta$ are parametrized by variables $(x_\beta)_{\beta \in \Phi(\delta)}$ living in $k^\times$ (whose for which $\delta_i = 1$) or $k$. For this specific subexpression, one can check that the first $(2n-2)$-th roots correspond to variables in $k^\times$ whereas the last $(n-1)$-th correspond to variables in $k$. 
Therefore, for $\mathbf{y}=(y_1,\ldots,y_{n}) \in (k^\times)^{n}$, we can consider the element of $D_\delta$  associated to the following specialization:
\centers{$(x_\beta)_{\beta \in \Phi(\delta)} \, = \, (y_1,y_2, \ldots, y_{n-1},y_{n},-y_{n-1}, \ldots, -y_3,-y_2, 0, \ldots, 0). $}

\noindent The corresponding representative in $U^*$ is thus given by

\centers{$ u_{\mathbf{y}} \, = \, u_{\beta_n}^{*}(y_1) u_{2\beta_1 + \beta_2 + \cdots + \beta_{n-1}}^* (y_2) \underbrace{u_{\beta_2}^*(y_3) \cdots u_{\beta_{n-1} + \beta_n}^*(y_n) \cdots u_{\beta_2}^*(-y_3)}_{\begin{array}{c} v_{\mathbf{y}} \end{array}}u_{2\beta_1 + \beta_2 + \cdots + \beta_{n-1}}^* (-y_2)$}

\noindent where, with a view of making the computations readable, we have denoted by $u_\alpha^* = u_{-\alpha}$ the one-parameter subgroup corresponding to the root $-\alpha$. By successive applications of formula \ref{for}.(i) and \ref{for}.(ii), the expression of $v_{\mathbf{y}}$ simplifies into 

\centers{$ v_{\mathbf{y}} \, = \, u_{\beta_2 + \cdots + \beta_n}^*(\pm y_3 \cdots y_n) \cdots u_{\beta_{n-2} +\beta_{n-1} + \beta_{n}}^*(\pm y_{n-1} y_n)u_{\beta_{n-1}+\beta_n}^*(y_n).$}

\noindent Now, by formula \ref{for}.(i) and \ref{for}.(iii) we get

\centers{$\begin{array}{r@{\ \, = \ \, }l} u_{\mathbf{y}} &  u_{\beta_n}^*(y_1) u_{2\beta_1 +  \cdots +2\beta_{n-1}+\beta_n}^* (\pm y_2 \cdots y_n) \, v_{\mathbf{y}} \\[5pt] & u_{\beta_n}^*(y_1) u_{2\beta_1 +  \cdots + 2\beta_{n-1} +\beta_n}^* (\pm y_2 \cdots y_n)  u_{\beta_2 + \cdots + \beta_n}^*(\pm y_3 \cdots y_n) \cdots u_{\beta_{n-1}+\beta_n}^*(y_n). \end{array}$}

\noindent Since every element of $V$ can be written in this form, this proves that $D_\delta$ contains the n-dimensional variety $\Omega$.

\sk 

\noindent (iii) As in (ii), it is easy to compute the sequence of roots occurring in the canonical expression in $U^*$ of the elements of $D_\gamma$ (see notations \ref{not}). Its opposite is given by

\centers{$-\Phi(\gamma) \, = \, \big( \hskip -1.3mm \begin{array}[t]{l}   \beta_n \, ; \beta_1 + \cdots + \beta_{n-1}\, ; \beta_2 + \cdots + \beta_{n-1}\, ; \ldots  ; \beta_{n-1} \, ; \\[3pt] 
 \beta_n \, ; \beta_1 + \cdots + \beta_{n-1}\, ; \beta_2 + \cdots + \beta_{n-1}\, ; \ldots  ; \beta_{n-1}\big). \end{array}$}

\noindent For $\mathbf{z} = (z_1, \cdots,z_n,t) \in (k ^\times)^{n+1}$, let us consider the representative $u_{\mathbf{z}} \in U^*$ of the element of $D_\gamma$ corresponding to the following choice of variables:

\centers{$(x_\beta)_{\beta \in \Phi(\delta)} \, = \, (z_n,z_1t,z_2t^2,z_3 t^2, \ldots,z_{n-1}t^2, t^{-2},-z_1t,-z_2t^2,-z_3 t^2, \ldots,-z_{n-1}t^2). $}

\noindent Because all the variables are non-zero, there is no need to check which root should correspond to a variable in $k^\times$ or $k$. Besides, we can apply formula \ref{for}.(i) to change the order of some terms in $u_{\mathbf{z}}$ and to get

\centers{$u_\mathbf{z} \, = \, u_{\beta_n}^*(z_n) u_{\beta_1+ \cdots + \beta_{n-1}}^* (z_1 t) \underbrace{\cdots u_{\beta_{n-1}}^*(z_{n-1} t^2) u_{\beta_n}^*(t^{-2}) u_{\beta_{n-1}}^*(-z_{n-1} t^2) \cdots}_{\begin{array}{c}v_{\mathbf{z}} \end{array}} u_{\beta_1+ \cdots + \beta_{n-1}}^* (-z_1 t).$} 

\noindent Applying successively formula \ref{for}.(i) and \ref{for}.(iv) leads to the following expression for $v_\mathbf{z}$

\centers{$v_\mathbf{z} \, = \, u_{\beta_n}^*(t^{-2}) u_{\beta_2 + \cdots + \beta_n}^*(\pm z_2) u_{\beta_3 + \cdots + \beta_n}^*(\pm z_3) \cdots u_{\beta_{n-1} + \beta_n}^*(\pm z_{n-1}).$}

\noindent Then, by using formula \ref{for}.(i) and then \ref{for}.(v) we obtain

\centers{$\begin{array}{r@{\, \ = \ \,}l} u_\mathbf{z} & u_{\beta_n}^*(z_n) \, \big[u_{\beta_1 + \cdots + \beta_{n-1}}^*(z_1t)\, ; u_{\beta_n}^*(t^{-2})\big] \, v_\mathbf{z} \\[5pt]
& u_{\beta_n}^*(z_n) u_{2\beta_1 + \cdots + 2\beta_{n-1}+\beta_n}^*(\pm z_1^2) u_{\beta_1 + \cdots + \beta_n}^*(\pm z_1 t^{-1}) v_\mathbf{z}. \\[5pt]
 \end{array}$} 

\noindent Finally, in this expression it is possible to evaluate the limit at $t=\infty$

\centers{$\displaystyle \lim_{t \to \infty} u_\mathbf{z} \, = \, u_{\beta_n}^*(z_n) u_{2\beta_1 + \cdots + 2\beta_{n-1}+\beta_n}^*(\pm z_1^2) u_{\beta_2 + \cdots + \beta_n}^*(\pm z_2) \cdots u_{\beta_{n-1} + \beta_n}^*(\pm z_{n-1}).$}

\noindent Once again, we observe that every element of $V$ can be written in this form, which proves that $\Omega = V\cdot B$ is contained in $\overline{D_\gamma}$.\end{proof}

\begin{cor} For any positive integer $n$, there exist $w \in W$, a reduced expression of $w$,  and $\gamma,\delta \in \Gamma_1$ two subexpressions of $w$ such that:

\begin{itemize}

\item[$\bullet$] $D_\delta \nsubseteq \overline{D_\gamma} $;

\item[$\bullet$] $\dim \overline{D_\gamma} \cap D_\delta \, \geq \,  n $.

\end{itemize}

\noindent In particular, this gives a negative answer to question \ref{q1}.

\end{cor}

\bk\mk

\noindent \textbf{II.3 - Disjointness of cells} \mk

We move now attention to the problem raised in question \ref{q2}. We assume that $n=3$ and we consider the following two distinguished subexpressions of $w_0$ associated to the reduced expression $w_0 = t_3 t_2 t_1 t_2 t_3 t_2 t_1 t_2 t_1$


\leftcenters{and}{$ \begin{array}[b]{r@{\, \ = \ \, }l} 
\sigma & (1, t_2, 1 ,t_2,1,t_2,t_1,1,t_1)\\[4pt]
\tau &      (1,t_2,t_1,1,1,t_2,1,t_2,t_1).
\end{array}$}

\noindent We have $\tau \preceq \sigma$, and  the corresponding cells are subvarieties of $B^* t_2 \cdot B$ of dimension 6. 

\begin{prop} The closure $\overline D_\sigma$ of $D_\sigma$ in the Schubert cell $Bw_0\cdot B$ is disjoint from the cell $D_\tau$, giving hence a negative answer to question \ref{q2}.
\end{prop}

\begin{proof} Using \cite[Section V.4.1]{Bou}, one can compute the one-parameter subgroups occurring in the canonical expression in $U^*$ of the elements of  $D_\sigma$ and $D_\tau$ (see notations \ref{not}). They are associated to the following sequences of roots:

\leftcenters{and}{$\begin{array}[b]{r@{\, \ = \ \,}l}
-\Phi(\sigma) & \big(  \beta_3 \, ; \beta_1+\beta_2 \, ; \beta_2\, ; \beta_3 \, ; 2 \beta_1 + \beta_2 \, ; \beta_1 + \beta_2  \big)  \\[4pt]
-\Phi(\tau) & \big(  \beta_3 \, ; 2\beta_1+\beta_2 \, ; \beta_2 + \beta_3\, ; \beta_1 \, ; 2 \beta_1 + \beta_2 \, ; \beta_1 + \beta_2 \big).  \end{array}$}

\noindent By definition, both of the cells $D_\sigma$ and $D_\tau$ are contained in $B^* t_2 \cdot B$, but since the simple negative root $-\beta_{1}$ does not occur in $\Phi(\sigma)$, the cell $D_\sigma$ is actually contained in $(B^* \cap {}^{t_{1}}B^*) t_2 \cdot B$, which is a closed subvariety of codimension 1 in $B^* t_2 \cdot B$. Therefore, the closure of $D_\sigma$ in the double Schubert cell $Bw_0\cdot B \cap B^* t_2 \cdot B$ is also contained in $(B^* \cap {}^{t_{1}}B^*) t_2 \cdot B$.

\sk

On the other hand, $-\beta_{1}$ occurs only once in $\Phi(\tau)$ and corresponds to a variable in $k^\times$: more precisely, if $i = 7$ then \\[2.5pt]
\indent $\bullet \ \tau^i =  t_2 t_1t_2$ and $\tau_i = 1$ so that $i \notin I(\tau)$ corresponds to a variable in $k^\times$; \\[2.5pt]
\indent $\bullet \ \tau^i(-\alpha_i) = \tau^i(-\beta_1) = t_2 t_{1} t_2(-\beta_1) = - \beta_{1}$
\\[2,5pt]
so that the cell $D_\tau$ is disjoint from  $(B^* \cap {}^{t_{1}}B^*) t_2 \cdot B$ and then from the closure of $D_\sigma$.\end{proof}

\begin{rmk} This situation is not specific to the low-dimensional cells. One can actually extend this example to the type B$_n$ for any $n \geq 3$ by considering the concatenation of $\sigma$ and $\tau$  with the subexpression of $v = t_n \cdots t_2 t_1 t_2 \cdots t_n$ defined by

\centers{$ \eta = (1, 1, \ldots, 1, t_2,1,t_2,1,\ldots,1). $}

\noindent The Deodhar cells corresponding to the subexpressions $\tilde \sigma = \eta \cdot \sigma$ and $\tilde \tau = \eta \cdot \tau$ are now of dimension $2n+2$, and satisfy indeed the previous proposition.
\end{rmk}

\end{document}